\documentclass[a4paper]{elsarticle}

\makeatletter
\def\ps@pprintTitle{%
  \let\@oddhead\@empty
  \let\@evenhead\@empty
  \let\@oddfoot\@empty
  \let\@evenfoot\@empty
}
\makeatother

\usepackage[utf8]{inputenc}
\usepackage[T1]{fontenc}
\usepackage[english, francais]{babel}
\usepackage{amsmath}
\usepackage{amssymb}
\usepackage{amsthm}
\usepackage{xcolor}
\usepackage{verbatim}
\usepackage{accents}
\usepackage{hyperref}
\usepackage{todonotes}

\newtheorem{theorem}{Theorem}[section]
\newtheorem{corollary}[theorem]{Corollary}
\newtheorem{lemma}[theorem]{Lemma}
\newtheorem{proposition}[theorem]{Proposition}

\theoremstyle{definition}
\newtheorem{remark}[theorem]{Remark}

\theoremstyle{definition}

\theoremstyle{definition}

\makeatletter
\def\dashint{\operatorname
{\,\,\text{\bf--}\kern-.98em\DOTSI\intop\ilimits@\!\!}}
\makeatother

\def\.5{\frac{1}{2}}

\def\bR{\mathbb{R}}

\def\cM{\mathcal{M}}

\def\cL{\mathcal{L}}

\journal{J. Math. Pures Appl.}
\begin{document}

\begin{frontmatter}

\title{Fabes-Stroock approach to higher integrability of Green's functions and ABP estimates with $L_d$ drift}

\author[PJ]{Pilgyu Jung}
\ead{pilgyujung@skku.edu}
\author[KW]{Kwan Woo}
\ead{kwan.woo@unibas.ch}

\address[PJ]{Department of Mathematics, Sungkyunkwan University, 2066 Seobu-ro, Jangan-gu, Suwon-si, Gyeonggi-do, 16419, Republic of Korea}
\address[KW]{Departement Mathematik und Informatik, Universit\"at Basel, CH-4051 Basel, Switzerland}

\begin{abstract}
\selectlanguage{english}
We explore the higher integrability of Green's functions associated with the second-order elliptic equation $a^{ij}D_{ij}u + b^i D_iu = f$ in a bounded domain $\Omega \subset \mathbb{R}^d$, and establish an enhanced version of Aleksandrov's maximum principle.
In particular, we consider the drift term $b=(b^1, \ldots, b^d)$ in $L_d$ and the source term $f \in L_p$ for some $p < d$.
This provides an alternative and analytic proof of a result by N. V. Krylov (\textit{Ann. Probab.}, 2021) concerning $L_d$ drifts.
The key step involves deriving a Gehring-type inequality for Green's functions by using the Fabes-Stroock approach (\textit{Duke Math. J.}, 1984).
\vskip 0.5\baselineskip

\selectlanguage{francais}
\noindent{\bf R\'esum\'e} \vskip 0.5\baselineskip \noindent
Nous étudions la sur-intégrabilité des fonctions de Green associées à l'équation elliptique du second ordre $a^{ij}D_{ij}u + b^i D_iu = f$ dans un domaine borné $\Omega \subset \mathbb{R}^d$, et établissons une version améliorée du principe du maximum d’Aleksandrov.
En particulier, nous considérons un terme de dérive $b=(b^1, \ldots, b^d)$ appartenant à $L_d$ et un second membre $f$ dans $L_p$ pour un certain $p < d$.
Nous proposons ainsi une démonstration analytique alternative d’un résultat de N.~V.~Krylov (\textit{Ann. Probab.}, 2021) concernant les dérives dans $L_d$.
L’étape clé consiste à établir une inégalité de type Gehring pour les fonctions de Green en utilisant l’approche de Fabes et Stroock (\textit{Duke Math. J.}, 1984).\end{abstract}
\selectlanguage{english}


\end{frontmatter}
\selectlanguage{english}

\section{Introduction}
Let $d \ge 2$ be a positive integer and $\bR^d$ be a $d$-dimensional Euclidean space with points $x = (x_1, \ldots, x_d)$.
Let $\delta \in (0, 1]$ and $a=[a^{ij}(x)]$ be a measurable $d \times d$ symmetric matrix-valued function satisfying an ellipticity condition: 
\begin{equation}
	\label{eq_elliptic}
\delta |\xi|^2\le a^{ij}(x) \xi_i\xi_j \le \delta^{-1}|\xi|^2,  \quad \forall \xi, \, x \in \bR^d.
\end{equation}
For a domain $\Omega \subset \bR^d$, let $b=(b^1, \ldots, b^d)$ be an $\bR^d$-valued  measurable function satisfying 
\[
\|b\|_{L_d(\Omega)} =\left(\int_{\Omega} |b(x)|^d\,dx\right)^{1/d} \le \|b\|
\]
for some non-negative fixed constant $\|b\|$.
We define
\[
\cL u =\cL_{a,b} u := a^{ij}D_{ij}u + b^iD_i u
\]
and let $\cL^* v = D_{ij}(a^{ij}v) - D_i(b^i v)$ be the formal adjoint of $\cL$.

This article addresses the \textit{higher integrability} of Green's functions for $\cL$ together with the following \textit{maximum principle}: 
\begin{equation}
    \label{classical abp}
\sup_{\Omega} u_+ \le \sup_{\partial\Omega} u_+ + N\|(\mathcal{L} u)_-\|_{L_p(\Omega)}.
\end{equation}
Here, $N$ depends only on $d$, $p$, $\delta$, $\operatorname{diam}(\Omega)$, and $\|b\|$.\\

The classical Aleksandrov-Bakel'man-Pucci (ABP) estimate, which corresponds to the case $p=d$ in \eqref{classical abp}, is a foundational result in the regularity theory of non-divergence form elliptic and parabolic equations.
This estimate, due to the work of Aleksandrov \cite{MR0199540}, Bakel'man \cite{MR0126604}, and Pucci \cite{MR0214905}, has been instrumental in various fields.
For instance, based on the ABP estimate, one can proceed with the Krylov-Safonov theorem, which is about the Harnack inequality and H\"older continuity of the solution to linear equation $a^{ij}D_{ij} u=f$ \cite{MR563790, MR2667641}.
For nonlinear equations, the regularity theory of viscosity solutions based on ABP estimates is well introduced in \cite{caffarelli1995fully}.
Furthermore, in the context of stochastic analysis, the ABP-type estimate also plays a crucial role in the proof of weak uniqueness of stochastic differential equations, see \cite{MR4317707}, for instance.

When there is no drift term, \textit{i.e.}, $b=0$, Fabes and Stroock \cite{MR771392} extended the estimate \eqref{classical abp} to the case $p>d-\varepsilon$ for some $\varepsilon>0$.
More precisely, by analyzing the local integrability of Green's functions and utilizing a Gehring-type Lemma, they raised the integrability beyond $d/(d-1)$.
This result is significant in $L_p$-theory for fully nonlinear equations, in particular, Escauriaza \cite{MR1237053} (see Winter \cite{MR2486925} for boundary estimates) used the Fabes-Stroock's result to extend Caffarelli's $W_{p}^2$ ($p > d$) estimates \cite{MR1005611} for solutions to a class of fully nonlinear equations to the case $p > d - \varepsilon$. 

On the other hand, for a general drift term $b\in L_d(\Omega)$, the corresponding extension of the Fabes-Stroock result remained an open problem until it was resolved by Krylov \cite{MR4252191}. 
Building on results in \cite{MR4252191}, the $L_p$-theory ($p > d - \varepsilon$) for fully nonlinear elliptic equations with $L_d$ drifts was further developed in \cite{Krylov2020}, leading to \eqref{classical abp} for $L_p$-viscosity solution in \cite{MR4515258}.
In particular, the author in \cite{MR4252191} obtained (among other various results on the stochastic process) the ABP estimate with $b\in L_d(\Omega)$ when $\mathcal{L} u \in L_p$ for some $p<d$ through a probabilistic approach such as sophisticated stopping-time arguments.
However, such probabilistic techniques may be challenging for readers unfamiliar with probability theory.
To make this significant result more accessible, we provide self-contained analytic proof in line with Fabes-Stroock \cite{MR771392}.

\subsection{\textbf{Fabes-Stroock approach}}
We briefly introduce the approach used in \cite{MR771392} to establish the higher integrability of Green's function.
The authors investigated the behavior of nonnegative solutions to adjoint equations $\cL^*_0 v = 0$ in $\Omega$ (see \eqref{eq240815_01} for the precise definition of adjoint problems), where $\cL_0 = a^{ij}D_{ij}$.
Their approach begins with the inequality
\begin{equation}
	\label{eq240811_01}
\left( \dashint_B v(y)^{d/(d-1)}\,dy \right)^{(d-1)/d} \le N(d, \delta) \dashint_B v(y)\,dy,
\end{equation}
which holds for any balls $B$ such that its concentric double $2B$ lies within $\Omega$.
This result is obtained by solving the equation $\cL_0 u= f$ for an arbitrary  $f \in L_d(\Omega)$, and then estimating
\[
\int_{B} v(y) f(y) \,dy = \int_B  v(y)  \cL_0 u(y)  \,dy \le N\int_{B} v(y)\,dy
\]
using the fact that $\cL_0^* v = 0$ and regularity properties of $u$.

On the other hand, classical theory provides lower bounds for the mean of Green's function $G$ associated with $\cL_0$, as given by 
\begin{equation}
	\label{eq240811_02}
\int_{B} {G}_{2B}(x, y) \,dy \ge \nu
\end{equation}
for some $\nu = \nu(d, \delta) > 0$.
Since ${G}_{\Omega} - {G}_{2B}$ ($2B \subset \Omega$) is a nonnegative solution to $\cL_0^*v = 0$, combining \eqref{eq240811_01} and \eqref{eq240811_02} yields the reverse H\"older inequality for ${G}_{\Omega}$:
\begin{equation}
	\label{eq240811_03}
\left( \dashint_{B} {G}_{\Omega}(x, y)^{d/(d-1)}\,dy \right)^{(d-1)/d} \le N(d, \delta)  \dashint_{B} {G}_{\Omega}(x, y)\,dy 
\end{equation}
for any $x \in \Omega$ and any ball $B$ such that its concentric quadruple $4B$ is contained in $\Omega$.
Finally, by \eqref{eq240811_03} and the $A_{\infty}$-weight theory \cite{10.1007/BF02392268, Coifman1974}, one can deduce that
\[
\sup_{x\in \Omega} \int_{\Omega}{G}_{\Omega}(x, y)^q \,dy \le N(d, \delta, \operatorname{diam}(\Omega))
\]
for some $q > (d-1)/d$.

For the case $b \neq 0$, Fok \cite{MR1632776} adapted the Fabes-Stroock approach to obtain an inequality similar to \eqref{eq240811_01} for nonlinear problems.
However, in this case, even for the linear operators $\cL u = a^{ij}D_{ij}u + b^iD_iu $, a slightly higher integrability of drift term is required (for example, $b \in L_{d+ \varepsilon}$) due to the low integrability of $Du$.
We also refer to~\cite{Cab1995}, which presents a refined ABP estimate for bounded drift $b$, where the domain dependence is captured via geometric quantities finer than just the diameter.\\

In this paper, to handle the critical drift $b \in L_d(\Omega)$, we employ Aleksandrov solutions $z \in W_{\infty}^1$ (see Lemma \ref{test estimate}) that satisfy $f \le \cL z$ instead of solving the linear equation $\cL u = f$.
Unlike $u$ of $\cL u = f$, it is difficult to expect regularity of $D^2z$ in our situation.
However, the analysis needed to obtain \eqref{eq240811_01} does not rely on the regularity of $D^2z$.
Instead, by leveraging the properties of the convex solution $z$, we obtain the boundedness of $|Dz|$ (generally $|Du|$ is not bounded), which can be used more effectively in the proof. See Proposition \ref{prop0931}.
We also prove a counterpart of \eqref{eq240811_02} that takes into account $b \in L_d(\Omega)$; see Lemma \ref{lem2} and Corollary \ref{cor4}.
Another fundamental idea we employ to obtain the aforementioned key estimate is the fact that the $L_d$-norm of $b$ is invariant under scaling $b(x) \mapsto rb(rx)$.
This property renders the effect of the drift $b$ negligible compared to the diffusion $[a^{ij}]$ as the size of the ball increases.
The advantage of this approach is that the required ball size can be determined solely by $d, \delta$ and $\|b\|$ using the arithmetic-geometric mean, making the \textit{a priori} estimate \eqref{classical abp} more practical.\\

This paper is organized as follows. 
In Section \ref{abp notation}, we introduce some notation and main results.
We state auxiliary estimates in section \ref{abp auxiliary}.
The main theorems are proved in section \ref{proof abp}.

\section*{Notation} Throughout this paper, we use the following notation.\\

\noindent We denote by $|x|$ the Euclidean norm of a vector $x \in \mathbb{R}^d$, and by $|A|$ the Lebesgue measure of a measurable set $A \subset \mathbb{R}^d$.
The same symbol $|\cdot|$ is used for both the norm and the measure, and the intended meaning will be clear from the context.
For any measurable set $A \subset \mathbb{R}^d$ with positive measure, we define the average of a function $f$ over $A$ by
$$
\dashint_A f := \frac{1}{|A|} \int_A f.
$$
For real numbers $a, b$, and $c$, we define
$$
a \vee b := \max\{a, b\}, \quad a \wedge b := \min\{a, b\}, \quad c_{+} := \max\{c, 0\}, \quad c_{-} := \max\{-c, 0\}.
$$
We adopt the summation convention over repeated indices and denote partial derivatives by $D_i = \partial / \partial x_i$ and $D_{ij} = D_i D_j$.
If a constant $N$ depends on parameters $d, \nu, \ldots$, we indicate this dependence by writing $N = N(d, \nu, \ldots)$.
\\

\noindent  Given a domain $\Omega \subset \mathbb{R}^d$, we write $\partial\Omega$ for its boundary and $\operatorname{diam}(\Omega)$ for its diameter, defined by
$$
\operatorname{diam}(\Omega) := \sup_{x, y \in \Omega} |x-y|.
$$
For $R > 0$, we denote the open ball of radius $R$ centered at $x$ by
$$
B_R(x) := \{ y \in \mathbb{R}^d : |y-x| < R \},
$$
and simply write $B_R := B_R(0)$ when centered at the origin.
Similarly, the open cube of side length $R$ centered at $x$ is denoted by
$$
Q_R(x) := \{ y \in \mathbb{R}^d : |x_i - y_i| < R/2 \quad \text{for all } i = 1, \ldots, d \},
$$
with $Q_R := Q_R(0)$ when centered at the origin.\\

\noindent We now introduce the function spaces used throughout the paper.
Let $\Omega$ be a domain in $\mathbb{R}^d$.
We denote by $C_c^{\infty}(\Omega)$ the space of infinitely differentiable functions with compact support in $\Omega$, and by $C(\overline{\Omega})$ the space of continuous functions on the closure $\overline{\Omega}$.
The space $C^\infty(\overline{\Omega})$ denotes the set of functions that are infinitely differentiable up to the boundary $\partial \Omega$.\\

\noindent  For $p \in [1, \infty)$, we denote by $L_p(\Omega)$ the usual Lebesgue space of $p$-integrable functions on $\Omega$, equipped with the norm $\| \cdot \|_{L_p(\Omega)}$.
An element $u \in L_{p, \mathrm{loc}}(\Omega)$ means that $u \in L_p(K)$ for every compact set $K \subset \Omega$.
We denote by $W_p^2(\Omega)$ the Sobolev space consisting of functions $u \in L_p(\Omega)$ whose weak derivatives $D_i u$ and $D_{ij} u$ also belong to $L_p(\Omega)$.
The norm on $W_p^2(\Omega)$ is given by
$$
\| u \|_{W_p^2(\Omega)} := \| u \|_{L_p(\Omega)} + \| Du \|_{L_p(\Omega)} + \| D^2 u \|_{L_p(\Omega)}.
$$
We similarly define $W_{p, \mathrm{loc}}^2(\Omega)$ as the space of functions that locally belong to $W_p^2(\Omega)$.

\section{Preliminaries and main results}
\label{abp notation}
We call $G_\Omega$ the \textit{Green's function of $\cL$ associated with $\Omega$} if 
\[
u(x)=\int_{\Omega}G_\Omega (x,y)f(y)\,dy   
\]
belongs to $W_{d, \mathrm{loc}}^2(\Omega)\cap C(\overline{\Omega})$ for  $f\in L_d(\Omega)$ and satisfies the equation 
\begin{equation}
    \label{defi_green}
    \begin{cases}
        \cL u\, (= a^{ij}D_{ij}u + b^iD_iu) = - f  \quad \text{in} \quad \Omega,\\
        u=0 \quad \text{on} \quad \partial\Omega.        
    \end{cases}
\end{equation}
Unless otherwise specified, solutions $u$ to equation \eqref{defi_green} are understood to lie in suitable function spaces where the equation holds almost everywhere in $\Omega$.
We emphasize that all the estimates derived in this paper are \textit{independent of the regularity} of the coefficients $a^{ij}$, $b^i$, and the boundary $\partial\Omega$.
This fact, together with an approximation argument, allows us to assume the existence of Green's functions in the proof of a main result, Theorem \ref{cor_main}, which presents a refined version of ABP estimate.\\

For completeness, in the following remark, we also include an argument for the existence and uniqueness of the Green's function, based on $L_p$-theory.
While several approaches are available in the literature, we adopt this method for its relevance to the present framework.

\begin{remark}[Unique existence of the Green's function]
	\label{rmk_Green}
If we assume that $a^{ij}$ are uniformly continuous and $b^i$ are bounded, then the Green's function associated with $\Omega$ exists and is unique, provided that $\Omega$ is bounded and $  \partial\Omega\in C^2$.
Indeed, by the classical $L_p$-theory (see, for instance, \cite[Chapter 9]{MR0737190} or \cite[Theorem 11.3.2]{MR2435520}), if $f \in L_d(\Omega)$, there exists a unique solution $\mathcal{R}f$ to equation \eqref{defi_green} such that
\[
\|\mathcal{R}f\|_{W_d^2(\Omega)} \leq N(d, p, a^{ij}, b^i, \Omega) \|f\|_{L_d(\Omega)},
\]
where $a^{ij}$, $b^i$, and $\Omega$ satisfy the aforementioned conditions.
Then, using Sobolev embedding, we have that $\mathcal{R}f \in C(\overline{\Omega})$ and
\[
\|\mathcal{R}f\|_{C(\overline{\Omega})} \leq N_0 \|f\|_{L_d(\Omega)},
\]
where $N_0 = N_0(d, p, a^{ij}, b^i, \Omega)$.
This implies that for each $x \in \Omega$, the resolvent mapping $f \in L_d(\Omega) \mapsto \mathcal{R}f(x)$ defines a bounded linear functional whose norm is bounded by $N_0$.
By Riesz's representation theorem, for  $x\in \Omega$, there exists a unique function $G_\Omega(x, \cdot) \in L_{d/(d-1)}(\Omega)$ such that
\[
\mathcal{R}f(x) = \int_\Omega G_\Omega(x, y) f(y) \, dy
\]
Therefore, the existence and uniqueness of $G_\Omega$ are guaranteed.

On the other hand, even when the coefficients and the domain boundary are irregular, one can still define  Green's functions using probability theory. 
For more details on probabilistic solutions, see \cite{MR1483890}.
\end{remark}

\begin{theorem}
	\label{cor_main}
Let $\Omega$ be a bounded domain. There is a constant $d_0 = d_0(d, \delta, \|b\|) \in (d/2, d)$ such that the following holds:
If $u \in W_{p, \mathrm{loc}}^2(\Omega) \cap C(\overline{\Omega})$ with $p \in [d_0, \infty)$, then
\begin{equation}
	\label{eq1013}
\sup_{\Omega} u_{+} \le \sup_{\partial\Omega} u_+ + N  \operatorname{diam}(\Omega)^{2-d/p} \| \left(  \cL u \right)_{-} \|_{L_{p}(\Omega)},
\end{equation}
where $N$ depends only on $d$, $\delta$, and $\|b\|$.
\end{theorem}

\begin{remark}[Neumann problem on half balls]
By employing even/odd reflection and Green's function representation, an estimate analogous to \eqref{eq1013} can be obtained even in the case of the Neumann boundary value problem on a half ball. More precisely, for $R>0$,  we have
\[
\sup_{B_R \cap \{x_1 > 0 \}} u_{+} \le \sup_{ \substack{|x| = R \\ x_1 > 0}} u_+ + 2R\sup_{ \substack{|x| < R \\ x_1 = 0} } \left( D_1u\right)_{-} + NR^{2 - d/d_0}  \| \left(  \cL u \right)_{-} \|_{L_{d_0}(B_R\cap\{x_1>0\})},
\]
where $N$ and $d_0$ are the constants from Theorem \ref{cor_main}.
See Remark \ref{adjoint_Neumann} for a discussion on the adjoint problems with Neumann boundary conditions.
\end{remark}

From Theorem \ref{cor_main}, we derive a Fabes-Stroock-type refined version of the integrability estimate for Green's functions.

\begin{theorem}\label{thm_main}
Let $\Omega$ be a bounded domain.
Suppose that the operator $\cL$ admits a Green's function $G_{\Omega}$ associated with $\Omega$.
Then there exist positive constants $N$ and $q>d/(d-1)$ depending only on $d$, $\delta$, and $\|b\|$ such that 
\begin{equation}
	\label{eq_main}
\sup_{x\in \Omega}\left(\int_{\Omega}G_{\Omega}(x,y)^{q}\,dy\right)^{1/q}\le N\operatorname{diam}(\Omega)^{2-d/p},
\end{equation}
where $p$ is the conjugate exponent of $q$, i.e., $1/p + 1/q = 1$.
\end{theorem}

Note that the results in \cite{MR2064654} depend on the profile of the function $b$, whereas the constant $N$ in Theorem \ref{thm_main} depends only on the upper bound $\|b\|$ of the norm $\|b\|_{L_d(\Omega)}$ (as well as on $d$ and $\delta$).
Hence, \eqref{eq_main} serves as a refined version of the classical estimates of Green's functions in Corollary \ref{cor1}.
On the other hand, the author in \cite{MR2064654} addressed a more general boundary value problem, specifically the oblique derivative problem.
We also mention the recent work \cite{DK2022}, which proves an Aleksandrov estimate under the more general assumption that the drift $b$ belongs to a Morrey class.
Note that applying the estimate in \cite{DK2022} to the case $b \in L_d$ yields a result that depends not only on the $L_d$-norm of $b$, but also on its profile.
\\

\section{Auxiliary estimates}
\label{abp auxiliary}
We begin this section by introducing the main ingredients used in proving main results.
The first is a well-known reverse H\"older inequality, which can be found in, for example, \cite{MR0717034}.
In the next proposition, $d(x) = \operatorname{dist}(x, \partial Q_1)$ for $x \in Q_1$.
For $f \in L_{1, \mathrm{loc}}$, $\cM f$ is the standard Hardy-Littlewood maximal function and 
\[
\cM_{r_0} f (x) := \sup_{0< r < r_0} \frac{1}{|B_r(x)|} \int_{B_r(x)} |f(z)| \,dz
\]
is the local maximal function for $r_0 \in (0, \infty)$.

\begin{proposition}
    \label{reverse}
    Let $1<q< \infty$. 
    Suppose that $0 \le g\in L_{q}(Q_1)$ and we have 
$$
\cM_{r_0d(x)} \left( g^q \right) \left( x \right)\le N_0 \left( \cM g \left(x \right) \right)^q
+\theta \, \cM \left( g^q \right)  \left( x \right)
$$
for almost every $x\in Q_1$, where $N_0 > 1$, $0\le \theta<1$, and $0<r_0<1$.\\

\noindent Then for  $p\in[q,q+\varepsilon)$, 
\begin{equation*}
    \left(\dashint_{Q_{1/2}}g^p \,dx \right)^{1/p}
    \le N\left(\dashint_{Q_1}g^q \,dx \right)^{1/q},
\end{equation*}
where $N$ and $\varepsilon$ are positive constants depending on $d,q,r_0,N_0$, and $\theta$.
\end{proposition}

\begin{proof}
    By making minor changes to the proof of \cite[Theorem V.1.2]{MR0717034}, we establish the desired inequality.
\end{proof}

In the following lemma, we construct and estimate auxiliary functions to be used to prove the main theorem.
By applying dilation and employing the techniques from the proof of \cite[Lemma 2.4]{MR4252191}, along with Aleksandrov's original work on Monge-Ampère equations \cite{A58}, we obtain the following result.

\begin{lemma}\label{test estimate}
Let $r\in (0,\infty)$.
For a nonnegative function $f \in C_c^{\infty}(B_{2r})$ with $\|f\|_{L_d(B_{2r})} \le 1$, there exists a non-positive convex (smooth) function $z^{\varepsilon}$ and a constant $\varepsilon_0 > 0$ such that the following holds:\\

\noindent For $\varepsilon\in(0, \varepsilon_0)$ and $x \in B_{2r}$, we have
\begin{equation}
    \label{sup z Dz}
    r^{-1}|z^\varepsilon|+|Dz^\varepsilon|\le N(d, \delta, \|b\|)
\end{equation}
and 
\begin{equation}
	\label{eq0930_01}
 a^{ij}D_{ij}z^{\varepsilon}+b^i D_iz^{\varepsilon}
 \ge
 d\sqrt[d]{\det a} \cdot f^{\varepsilon} - I^\varepsilon,
\end{equation}
where $f^{\varepsilon} = f + \varepsilon$.
Here, $I^\varepsilon$ denotes an error term arising from the approximation procedure, which satisfies
\begin{equation}
\label{I epsilon}
 \lim_{\varepsilon\to0} \|I^\varepsilon\|_{L_d(B_{2r})} = 0.
\end{equation}
\end{lemma}

The existence of $z^{\varepsilon}$ satisfying \eqref{sup z Dz}--\eqref{I epsilon} is essentially due to Aleksandrov \cite{A58}, and the corresponding statement can be found in \cite[Theorem 2.3]{MR4252191}.
Nevertheless, we would like to provide a brief description in a modern language for the reader's convenience.

\begin{proof}[\textbf{Proof of Lemma \ref{test estimate}.}]
We set $f_r^{\varepsilon}(x) := f_r(x) + \varepsilon$, where $f_r = f(rx) \in C_c^{\infty}(B_2)$.
We also set $a_r(x) = a(rx)$, $b_r(x) = rb(rx)$, $a_r^{\varepsilon} = a_r^{(\varepsilon)}$, and $b_r^{\varepsilon} = (|b_r|1_{B_2})^{(\varepsilon)} + \varepsilon$, where $g^{(\varepsilon)}$ means a standard mollification of $g$.
Recall that we have assumed that $a^{ij}$ is defined on $\mathbb{R}^d$ in \eqref{eq_elliptic}. 
(This is just for convenience since any elliptic coefficients defined on arbitrary domains can be extended to the whole space while preserving the ellipticity condition.)
We proceed in a few steps.\\

\noindent \emph{Auxiliary convex solutions.}
By classical regularity theory for Monge-Amp\'ere equations (see, for instance, \cite[Theorem 1.2]{TU83} or \cite[Theorems 1.1 and 7.1]{CNS84}), there exist unique \textit{non-positive smooth convex} functions $z_{1, r}^{\varepsilon}$ and $z_{2, r}^{\varepsilon}$ in $\overline{B_3}$ such that
\begin{equation}
	\label{eq0715_01}
\begin{aligned}
\sqrt[d]{\det D^2 z_{1, r}^{\varepsilon}} &= f_r^{\varepsilon}  \quad \textrm{and} \\  \sqrt[d]{\det D^2 z_{2, r}^{\varepsilon}} &= \left( d \sqrt[d]{\det a_r^{\varepsilon}} \right)^{-1}   \left(1 + (\varepsilon+|Dz_{2, r}^{\varepsilon}|^2 )^{1/2} \right)    \cdot b_r^{\varepsilon}
\end{aligned}	
\end{equation}
in $B_3$ and $z_{1, r}^{\varepsilon} = z_{2, r}^{\varepsilon} = 0$ on $\partial B_3$.
We now define
\[
z^{\varepsilon}(x) = r^2z_r^{\varepsilon} \left(x/r \right), \quad \textrm{where}\quad z_r^{\varepsilon} = z_{1, r}^{\varepsilon} + N_1r^{-1}z_{2, r}^{\varepsilon}.
\]
We shall establish the $W_{\infty}^1$ estimate \eqref{sup z Dz}, the comparison inequality \eqref{eq0930_01}, and the error estimate \eqref{I epsilon}, based on the properties of convex functions $z_{1,r}^{\varepsilon}$ and $z_{2,r}^{\varepsilon}$.\\

\noindent \emph{$W_{\infty}^1$ estimates.}
First, note that $z_{1, r}^{\varepsilon}$ and $z_{2, r}^{\varepsilon}$ enjoy the Aleksandrov-type maximum principle (\textit{cf}. \cite[Ch. 9]{MR0737190})
\[
|z_{1, r}^{\varepsilon}| \le N_1(r^{-1} + \varepsilon)  \quad \textrm{and} \quad |z_{2, r}^{\varepsilon}| \le N_2 + N_1\varepsilon,
\]
where $N_1=N_1(d)$ and $N_2=N_2(d,\delta,\|b\|)$.
Here, we use the assumption that $\|f\|_{L_d(B_{2r})} \le 1$ in the first inequality.

To estimate $Dz_{1,r}^\varepsilon$ and $Dz_{2,r}^\varepsilon$ in $B_2$, we make use of  the convexity and non-positivity of $z_{1, r}^{\varepsilon}$ and $z_{2, r}^{\varepsilon}$, together with the fact that $z_{1,r}^\varepsilon = z_{2,r}^\varepsilon = 0$ on $\partial B_3$. 
More precisely, to estimate $|Dz_{1,r}^\varepsilon(x)|$ for $x \in B_2$, suppose that $|D z_{1,r}^\varepsilon(x)| \neq 0$.
Let $e = D z_{1,r}^\varepsilon(x) / |D z_{1,r}^\varepsilon(x)|$ and $t_0$  be a positive number such that $x + t_0 e \in \partial B_3$.
We define $g(t) = z_{1,r}^\varepsilon(x + t e)$.
Then, by convexity and the estimate on $|z_{1,r}^\varepsilon|$, we obtain
\begin{equation}
	\label{eq_250801_1}
|D z_{1,r}^\varepsilon(x)| = g'(0) \le \frac{g(t_0) - g(0)}{t_0} = \frac{-g(0)}{t_0} \le \frac{|z_{1,r}^\varepsilon(x)|}{3 - |x|} \le N_1 (r^{-1} + \varepsilon).
\end{equation}
A similar argument shows that $|D z_{2,r}^\varepsilon(x)| \le N_2 + N_1 \varepsilon$ for $x \in B_2$.
Then, together with the above estimates, it is straightforward to check that $z^{\varepsilon}$ satisfies \eqref{sup z Dz} for sufficiently small $\varepsilon$.\\

\noindent \emph{Comparison and error estimate.}
It remains to show \eqref{eq0930_01} and \eqref{I epsilon}.
Note that 
\begin{align*}
a^{ij}_rD_{ij}z_{1, r}^{\varepsilon} + b^i_r D_i z_{1, r}^{\varepsilon} &\ge d   \sqrt[d]{\det a_r}  \sqrt[d]{\det D^2 z_{1, r}^{\varepsilon}} - |b_r | |D z_{1, r}^{\varepsilon}| \\
&= d   \sqrt[d]{\det a_r} \cdot f_r^{\varepsilon} - |b_r | |D z_{1, r}^{\varepsilon}| \\
&\ge d   \sqrt[d]{\det a_r} \cdot f_r^{\varepsilon} - |b_r | N_1(r^{-1}+\varepsilon)
\end{align*}
in $B_2$, where the last inequality is due to \eqref{eq_250801_1} in $ B_2$.
Similarly, we have
\begin{align*}
a^{ij}_rD_{ij}z_{2, r}^{\varepsilon} + b_r^i D_i z_{2, r}^{\varepsilon} &\ge d   \sqrt[d]{\det a_r}  \sqrt[d]{\det D^2 z_{2, r}^{\varepsilon}} - |b_r | |D z_{2, r}^{\varepsilon}| \\
&=  \left( \frac{\det a_r}{\det a_r^{\varepsilon}}  \right)^{1/d} \cdot   \left(1 +(\varepsilon+ |Dz_{2, r}^{\varepsilon}|^2)^{1/2} \right)   \cdot b_r^{\varepsilon} \, - \, |b_r | |D z_{2, r}^{\varepsilon}|.
\end{align*}
Thus, from the definition $z_r^{\varepsilon} = z_{1, r}^{\varepsilon} + N_1r^{-1}z_{2, r}^{\varepsilon}$, we have
\begin{equation}
	\label{eq250801_02}
a^{ij}_rD_{ij}z_{ r}^{\varepsilon} + b^i_r  D_i z_{ r}^{\varepsilon} \ge d   \sqrt[d]{\det a_r} \cdot f_r^{\varepsilon} - I_{r}^{\varepsilon} \quad \textrm{in} \,\, B_2,
\end{equation}
where
\begin{align*}
I_{r}^{\varepsilon} &:=   |b_r | N_1( r^{-1} + \varepsilon)
- N_1r^{-1}\left( \frac{\det a_r}{\det a_r^{\varepsilon}}  \right)^{1/d} \cdot   \left(1 +(\varepsilon+ |Dz_{2, r}^{\varepsilon}|^2)^{1/2} \right)     \cdot b_r^{\varepsilon} \\
 &+  N_1r^{-1}|b_r | |D z_{2, r}^{\varepsilon}|\\
 &= N_1r^{-1}|b_r | \cdot (1 + |D z_{2, r}^{\varepsilon}|) - \left( \frac{\det a_r}{\det a_r^{\varepsilon}}  \right)^{1/d} \cdot N_1r^{-1}b_r^{\varepsilon} \cdot    \left(1 +(\varepsilon+ |Dz_{2, r}^{\varepsilon}|^2)^{1/2} \right)     \\
 &+ N_1 |b_r| \varepsilon. 
\end{align*}
Recall that $b_r^{\varepsilon} = (|b_r|1_{B_2})^{(\varepsilon)} + \varepsilon$ and then, by dominated convergence theorem with the help of the bound $|Dz_{2, r}^{\varepsilon}| \le N_2 + N_1 \varepsilon$ in $B_2$, we see that
\[
\| I_r^{\varepsilon}\|_{L_d(B_2)} \to 0 \quad  \textrm{as} \quad \varepsilon \to 0.
\]
Finally, by using the relation $z^{\varepsilon}(x) = r^2 z_r^{\varepsilon}(x/r)$, $a(x) = a_r(x/r)$, $b(x) = r^{-1}b_r(x/r)$, and $f^{\varepsilon}(x) = f_r^{\varepsilon}(x/r)$, we obtain \eqref{eq0930_01} for $x \in B_{2r}$ together with the error estimate \eqref{I epsilon} for $I^{\varepsilon}(x) := I_r^{\varepsilon}(x/r)$.
The lemma is proved.
\end{proof}

We end this section by presenting the classical ABP estimate which is read as the following theorem.
\begin{theorem}\label{thm_abp}
Let $\Omega$ be a bounded domain.   Then for any $u\in W_{d,\mathrm{loc}}^2(\Omega) \cap C(\overline{\Omega})$,
    \begin{equation}
    \label{eq1001}
        \sup_{\Omega}u_{+}\le \sup_{\partial \Omega}u_+ +N\operatorname{diam}(\Omega)  \|(\cL u)_{-}\|_{L_d(\Omega)},
    \end{equation}
where $N=N(d,\delta,\|b\|)$.
\end{theorem}
Using \eqref{eq1001}, one verifies that the Green's function $G_{\Omega}$ is nonnegative, \textit{e.g.}, the proof of Lemma \ref{lem3}.
Moreover, by using duality, we obtain the corollary below.

\begin{corollary}\label{cor1}
Suppose that $\cL$ admits a Green's function $G_{\Omega}$ associated with $\Omega$.
Then we have
\begin{equation*}
    \sup_{x\in \Omega}\left(\int_{\Omega}G_\Omega(x,y)^{d/(d-1)}\,dy\right)^{(d-1)/d}\le N \operatorname{diam}(\Omega),
\end{equation*}
where $N=N(d,\delta,\|b\|)$.
\end{corollary}

\section{Proofs of the main results}
	\label{proof abp}
Our strategy for proving the main results, Theorems~\ref{cor_main} and~\ref{thm_main}, is as follows.  
In Section~\ref{sec_key}, we derive a key estimate (Proposition~\ref{prop1002}) by combining Proposition~\ref{prop0931} and Corollary~\ref{cor4}, in order to apply a Gehring-type lemma.
This estimate is then used in Section~\ref{proof_main_0419} to establish the higher integrability of the Green's function (see~\eqref{eq_0418_1}).  
Finally, the main results are deduced from~\eqref{eq_0418_1}.\\

\subsection{\textbf{Key estimate for Gehring-type lemma}}
	\label{sec_key}

We start with introducing a notion of solution to the adjoint operator $\cL^*$.
We say $v\in L_{1,\mathrm{loc}}(\Omega)$ is a nonnegative solution of an adjoint problem $\cL^* v=0$ if $v\ge0$, and 
\begin{equation}	
	\label{eq240815_01}
\int_\Omega v \cL u \, dx = 0
\end{equation}
for all $u\in C_c^{\infty}(\Omega)$ with $u\ge0$.\\

We establish the following estimate for nonnegative homogeneous solutions of the adjoint operator $\cL^*$.
\begin{proposition}
\label{prop0931}
Let $v$ be a nonnegative solution of $\cL^* v=0$ in $\Omega$ such that $v^{d/(d-1)} \in L_{1, \mathrm{loc}}(\Omega)$.
There exist positive constants $N$ and $K  (\ge 1)$ depending only on $d,\delta$, and $\|b\|$ such that 
\begin{equation*}
\dashint_{B_R(x_0)}v(y)^{d/(d-1)}\,dy \le N \left( \dashint_{B_{KR}(x_0)}v(y)\,dy \right)^{d/(d-1)}
\end{equation*}
\begin{equation*}
    +\frac{1}{8}\dashint_{B_{KR}(x_0)}v(y)^{d/(d-1)}\,dy,
\end{equation*}
where  $x_0 \in \Omega$ and $R\in(0,\infty)$ satisfy $B_{KR}(x_0) \subset \Omega$.
\end{proposition}

\begin{proof}
By translation, we may assume that $x_0 = 0$.
Let $r$ be a positive number such that $B_{4r}\subset \Omega$.
For a nonnegative function $f \in C_c^{\infty}(B_{2r})$ with $\|f\|_{L_d(B_{2r})} \le 1$, we take $z$ from Lemma \ref{test estimate}.
Let $\varphi$ be a nonnegative smooth function such that $\varphi(x)=1$ if $|x|\le r $ and $\varphi(x)=0$ if $|x|>3r/2$ with 
\begin{equation}
    \label{cut off estiamtes}
|D\varphi|\le \frac{N}{r}I_{B_{2r}\setminus B_r} \quad \text{and} \quad |D^2\varphi|\le \frac{N}{r^{2}}I_{B_{2r}\setminus B_r}. 
\end{equation}
Then, by  \eqref{eq0930_01} of Lemma \ref{test estimate}, we see that for any $\varepsilon\in(0,r)$,
\[
\int_{B_r}v(y)f^\varepsilon(y)\,dy 
\le
N \int_{B_r}v(y)(\cL z^\varepsilon(y)+I_\varepsilon(y))\,dy
\]
\[
\le 
N \int_{B_{3r/2}}v(y) \left( \cL z^\varepsilon(y) + I_\varepsilon (y) \right)  \varphi(y)\,dy
= N \int_{B_{3r/2}}v(y)\cL(z^\varepsilon \varphi)(y) \,dy
\]
\[
+ N\int_{B_{3r/2}}v(y)\left(\cL z^\varepsilon(y)\varphi(y)-\cL(z^\varepsilon \varphi)(y) \right)\,dy
 + N \int_{B_{3r/2}}v(y) I_\varepsilon (y)\varphi(y)\,dy
\]
\[
\le N \int_{B_{2r}}v(y)| I_\varepsilon (y)| \,dy 
+ Nr^{-1}\int_{B_{2r}}v(y)|Dz^\varepsilon(y)| \,dy
\]
\[
+Nr^{-2}\int_{B_{2r}}v(y)|z^\varepsilon(y)| \,dy
+Nr^{-1}\int_{B_{2r}\setminus B_{r}}v(y)|b(y)||z^\varepsilon(y)| \,dy,
\]
where we use $\eqref{cut off estiamtes}$ and the fact that $\cL^*v = 0$ in $\Omega$ to obtain the last inequality.
By applying the pointwise estimates \eqref{sup z Dz} to the above inequalities and H\"older's inequality, we reach to
\[
\int_{B_r}v(y)f^\varepsilon(y)\,dy
\le N \int_{B_{2r}}v(y)| I_\varepsilon (y)| \,dy +Nr^{-1}\int_{B_{2r}}v(y)\,dy
\]
\[
+N\int_{B_{2r}\setminus B_r}v(y)|b(y)|\,dy
\]
\[
 \le Nr^{-1}\int_{B_{2r}}v(y)\,dy
+ N \left( \|b\|_{L_d(B_{2r}\setminus B_{r})} + \| I_{\varepsilon} \|_{  L_d(B_{2r}) } \right)\|v\|_{L_{d/(d-1)}(B_{2r})}.
\]
By taking $\varepsilon \to 0$ to the above estimates and utilizing \eqref{I epsilon}, it holds that
\[
\int_{B_{r}}v(y)f(y)\,dy
=\lim_{\varepsilon \to 0} \int_{B_{r}}v(y)f^{\varepsilon}(y)\,dy
\]
\[
\le Nr^{-1}\int_{B_{2r}}v(y)\,dy
+ N \|b\|_{L_d(B_{2r}\setminus B_{r})}\|v\|_{L_{d/(d-1)}(B_{2r})}
\]
Since the above inequalities hold for any positive $f \in C_c^{\infty}(B_{2r})$ with $\| f\|_{L_d} \le 1$, we have
\begin{equation} \label{thm2_eq1} 
\dashint_{B_{r}}v^{d/(d-1)}\,dy
\le N\left(\dashint_{B_{2r}}v\,dy\right)^{d/(d-1)}
+ N \|b\|_{L_d(B_{2r}\setminus B_{r})}^{d/(d-1)}\dashint_{B_{2r}}v^{d/(d-1)}\,dy.   
\end{equation}

For $n=0,1,2, \ldots$, take $r=R_n:=2^{n}R$ and let $\Gamma_n=B_{R_{n+1}}\setminus B_{R_n}$, 
\[
A_n = \dashint_{B_{R_n}}v^{d/(d-1)}\,dy, \quad B_n=\|b\|_{L_d(\Gamma_{n})}^{d/(d-1)}, \quad C_n=\left(\dashint_{B_{R_{n+1}}}v\,dy\right)^{d/(d-1)}.
\]
Then we rewrite \eqref{thm2_eq1} as
\begin{equation*}
    A_n\le {N}_0 C_n + N_0B_nA_{n+1}, \quad \forall n = 0, 1, \cdots,
\end{equation*}
where ${N}_0=N_0(d,\delta,\|b\|)$.
We apply the above inequality recursively to see that

\begin{equation*}
\begin{split}
    A_0&\le 
    \sum_{n=0}^{m-1}N_0^{n+1}\left(\prod_{k=0}^{n-1}B_k\right) C_n
    +N_0^{m}\left(\prod_{n=0}^{m-1}B_n\right) A_{m}
    \\
    &\le  
     NC_{m-1}
     +N_0^{m}\left(\prod_{n=0}^{m-1}B_n\right) A_{m},
\end{split}
    \end{equation*}
    where $N = N(d, \delta, \|b\|, m)$.
By the inequality of arithmetic and geometric means, we see 
\[
N_0^{m}\left(\prod_{n=0}^{m-1}B_n\right)\le \left(m^{-1}N_0^{d-1}\|b\|^d \right)^{m/(d-1)}.
\]
Take sufficiently large $m=m(d,\delta,\|b\|)$, so that the last expression of the above inequality is less than $1/8$. 
Thus, if $B_{2^mR} \subset \Omega$, we have that 
\begin{equation*}
    \dashint_{B_{R}}v^{d/(d-1)}\,dy
\le N\left(\dashint_{B_{2^m R}}v\,dy\right)^{d/(d-1)}
+ \frac{1}{8}\dashint_{B_{2^m R}}v^{d/(d-1)}\,dy,   
\end{equation*}
    where $N = N(d, \delta, \|b\|, m)$.
By taking $K=2^m$, the proposition is proved.
\end{proof}

\begin{remark}
	\label{adjoint_Neumann}
The proof of Proposition \ref{prop0931} can be applied to solutions to adjoint problems $\cL^*v = 0$ on a half ball with homogeneous Neumann boundary conditions on the flat boundary, that is, to $v$ satisfying
\[
\int_{B_r \cap \{x_1 > 0\}} v \cL u = 0
\]
for any $u \in C^{2}( \overline{B_r \cap \{x_1 > 0\}})$ with compact support in $B_r \cap \{x_1 \ge 0\}$ and $D_1u = 0$.
Indeed, by Lemma \ref{test estimate} with considering proper even/odd extensions of $f$, $a^{ij}$, and $b^i$ from $B_r \cap \{x_1 \ge 0\}$ to $B_r$, we have a non-positive convex function $z^{\varepsilon}$ defined on $B_r$ satisfying \eqref{sup z Dz}--\eqref{I epsilon}.
Then by the uniqueness of solutions to Monge-Amp\'ere equations \eqref{eq0715_01} with an even symmetry property of determinant of Hessian (\textit{i.e.}, $\det D^2f  (x) =\left( \det D^2g \right) \left( -x_1, x' \right)$ for $f(x) = g(-x_1, x')$, $x' \in \mathbb{R}^{d-1}$), one immediately verify that $z^{\varepsilon}$ is even in $x_1$, that is, $D_1z^{\varepsilon} = 0$ on $B_r \cap \{x_1 = 0\}$.
This enables us to follow the argument in the proof of Proposition \ref{prop0931} for $v$ above.
We refer the reader to \cite[Section 4]{MR2064654} for discussions on adjoint solutions to oblique derivative problems.
\end{remark}

The following Lemmas \ref{lem1001}, \ref{lem2} are needed to obtain Corollary \ref{cor4}, a kind of lower bound of Green's function associated with $B_{R}$. 
We refer \cite[Corollary 2.12]{MR4252191} for a probabilistic counterpart of Corollary \ref{cor4}.

\begin{lemma}\label{lem1001}
    There exist positive constants $S_0=S_0(d,\delta)$ and $\nu=\nu(d,\delta)$ such that if $\| b \| \le S_0$ and $u\in W_d^2(B_{1})$ satisfies $\cL u=-1$ in $B_1$ and $u \ge 0$ on $\partial B_1$, then $u(x)\ge \nu$ in $B_{1/2}$.
\end{lemma}

\begin{proof}
Let $S_0 \le 1$ be a positive number to be specified later.
Take a nonnegative smooth $\zeta$ such that $\zeta(x) = 1$ for $|x| \le 1/2$ and $\zeta(x) = 0$ for $|x| \ge 1$.
Then we have 
\begin{equation*}
    \cL \zeta = a^{ij}D_{ij}\zeta+b^iD_i\zeta \ge -N_0I_{\Gamma}-N_0|b|I_{\Gamma} \quad \textrm{in} \quad B_1,
\end{equation*}
where $\Gamma={B_{1}\setminus B_{1/2}}$ and $N_0$ is a positive constant depending on $d, \delta$.
Set $v=\zeta/N_0$.
Then, we see that 
\[
\cL v-\cL u\ge -|b|I_{\Gamma} \quad \textrm{in} \quad B_1,
\]
which implies that
\[
\left(\cL v -\cL u\right)_{-}\le |b|I_{\Gamma}  \quad \textrm{in} \quad B_1.
\]
Using Theorem \ref{thm_abp} with the fact that $u \in W_d^2(B_1) \subset C(\overline{B}_1)$ and $(v - u)_{+} = 0$ on $\partial B_1$, we see that
\[
\frac{1}{N_0} - u(x) = v(x)-u(x)\le N_{1}\left(\int_{\Gamma}|b|^d\right)^{1/d}, \quad \forall x \in B_{1/2},
\]
that is,
\[
\frac{1}{N_0} - N_{1} S_0 \le u(x), \quad \forall x \in B_{1/2},
\]
where $N_1$ is a positive constant depending only on $d, \delta$.
Hence, it suffices to take $S_0 := 1 \wedge (2N_0 N_1)^{-1}$.
The lemma is proved.
\end{proof}

\begin{lemma}\label{lem2}
    There exists  positive constants $R_0=R_0(d,\delta,\|b\|) > 2$ and $\nu=\nu(d,\delta)$ such that if $u\in W_d^2(B_{R_0})$ satisfies $\cL u=-1$ in $B_{R_0}$ and $u \ge 0$ on $\partial B_{R_0}$, then $u\ge \nu$ in $B_1$.
\end{lemma}

\begin{proof}
Take $m =1 +  [(\|b\|/S_0)^d]$ (that is, $m$ is the least positive integer greater than $(\|b\|/S_0)^d$), where $S_0$ is taken from Lemma \ref{lem1001}. Let $R_0 = R_0(d, \delta, \|b\|) :=2m + 2$ and for $i = 1, \cdots, m$, set $\Gamma^{(i)} := B_{2i+2} \setminus B_{2i}$.
Since
$$
\min_{1\le i\le m} \left(\int_{\Gamma^{(i)}}|b|^d\,dx\right)^{1/d}
\le\left(\frac{1}{m}\sum_{i=1}^{m}\int_{\Gamma^{(i)}}|b|^d\,dx\right)^{1/d}
\le m^{-1/d}\|b\| < S_0,
$$
there is at least one $i_0 \in \{1, 2, \cdots, m \}$ such that $\|b\|_{L_d(\Gamma^{(i_0)})} \le S_0$.

By applying Theorem \ref{thm_abp} to $-u  \in W_d^2(B_{R_0}) \subset C(\overline{B_{R_0}}) $, we have $u \ge 0$ in $B_{R_0}$.
Also note that $B_1(y) \subset \Gamma^{(i_0)}$ and $\|b\|_{L_d(B_1(y))} \le S_0$ for any $y$ satisfying $|y| = 2i_0 + 1$.
Then it follows from Lemma \ref{lem1001} that there is a positive number $\nu = \nu(d, \delta)$ such that $u \ge \nu$ in $B_{1/2}(y)$ for any $y \in \partial B_{2i_0 + 1}$.
Finally, by applying Theorem \ref{thm_abp} to $\nu - u$, we obtain $u \ge \nu$ in $B_{2i_0 + 1}$, which implies the desired result.
\end{proof}

By using a scaling $u_R(x)= (R_0/R)^2u(R/R_0 x)$, we derive the following from Lemma \ref{lem2}.

\begin{corollary}\label{cor1003}
There exists a positive constant $\nu=\nu(d,\delta,\|b\|)$ such that for any $R\in (0,\infty)$ if $u\in W_d^2(B_{R})$ satisfies $\cL u=-1$ in $B_R$ and $u \ge 0$ on $\partial B_{R}$, then $u(0)\ge \nu R^2$.
\end{corollary}

The following lemma refers to the intrinsic properties of Green's functions.

\begin{lemma}\label{lem3}
Let $\Omega$ and $\Omega'$ be bounded domains such that $\overline{\Omega'}\subset \Omega$.
Suppose that the operator $\cL$ admits Green's functions $G_{\Omega}$ and $G_{\Omega'}$ associated with $\Omega$ and $\Omega'$, respectively.
Then, 
\begin{enumerate}
\item for $x\in \Omega'$, $G_\Omega(x,\cdot)-G_{\Omega'}(x,\cdot)$ is a nonnegative solution of $\cL^* v=0$ in $\Omega'$.

\item for $x\in \Omega\setminus \Omega'$, $G_\Omega(x,\cdot)$ is also a nonnegative solution of $\cL^* v=0$ in $\Omega'$. 
\end{enumerate}
\end{lemma}
\begin{proof}
\textit{(1)} First, we show that for any $x\in \Omega'$, $G_\Omega(x,y)-G_{\Omega'}(x,y)\ge 0$ for almost every $y\in \Omega'$. 
Let $f$ be a nonnegative  function in $C_c^{\infty}(\Omega')$. 
Set 
\begin{equation*}
    u(x)=\int_{\Omega} G_{\Omega}(x,y)I_{\Omega'}f(y)\,dy,\quad
    w(x)=\int_{\Omega'} G_{\Omega'}(x,y)I_{\Omega'}f(y)\,dy.    
\end{equation*}
Under the assumption that $\mathcal{L}$ admits Green's functions $G_{\Omega}$ and $G_{\Omega'}$, and by the definition of the Green's function together with classical $L_p$-theory (refer to Remark~\ref{rmk_Green}), we see that $u$ and $w$ are $W_d^2(\Omega)\cap C(\overline{\Omega})$ solutions to
\begin{equation*}
    \cL u=-fI_{\Omega'} \quad \text{in} \quad \Omega, \quad u=0 \quad \text{on} \quad \partial \Omega
\end{equation*}
and 
\begin{equation*}
    \cL w=-f\quad \text{in} \quad \Omega', \quad w=0 \quad \text{on} \quad \partial \Omega',
\end{equation*}
respectively.
By applying Theorem \ref{thm_abp} to $w-u$, we obtain that  for $x \in \Omega'$,
\[
    w(x) - u(x) \le \sup_{\partial \Omega'}(w-u)_{+}+N\operatorname{diam}(\Omega')\| \left( \cL \left( w-u \right) \right)_{-}\|_{L_d(\Omega')}\]
    \[=\sup_{\partial \Omega'}(w-u)_{+} = 0.
\]
Therefore, we see that for any nonnegative $f \in C_c^{\infty}(\Omega')$,
\begin{equation*}
    \int_{\Omega'} G_{\Omega'}(x,y)f(y)\,dy\le
    \int_{\Omega'} G_{\Omega}(x,y)f(y)\,dy.
\end{equation*}
This implies that $G_\Omega(x, \cdot)-G_{\Omega'}(x, \cdot)$ is nonnegative almost everywhere in $\Omega'$. 

To show that $G_\Omega(x,\cdot)-G_{\Omega'}(x,\cdot)$ is a solution of $\cL^* v=0$ in $\Omega'$ for $x\in \Omega'$, we verify
\begin{equation}
	\label{lem3_eq1}
\int_{\Omega'} \left(G_\Omega \left( x,y \right) -G_{\Omega'} \left( x,y \right)\right) \cL \varphi \left( y \right)\,dy=0
\end{equation}
for any nonnegative $\varphi\in C_c^\infty(\Omega')$.
Let 
\[
v_1(x)=\int_{\Omega'} G_{\Omega'}(x,y)\cL \varphi(y)\,dy.
\]
Then $v_1$ satisfies 
\[
\cL v_1=-\cL \varphi , \quad v_1=0 \quad \text{on} \quad \partial \Omega'.
\]
Since $\varphi$ is also zero on $\partial\Omega'$, $v_1=-\varphi$ in $\Omega'$ by Theorem \ref{thm_abp}.
Likewise, if we set $v_2$ as follows,
\[
v_2(x)=\int_{\Omega'}G_{\Omega}(x,y)\cL \varphi(y)\,dy=\int_{\Omega}G_{\Omega}(x,y)\cL \varphi(y)\,dy,
\]
then $v_2=-\varphi$ in $\Omega$.
This implies that $v_1=v_2$ in $\Omega'$, so that the equation \eqref{lem3_eq1} holds true.\\

\noindent \textit{(2)} Moreover, by the observation $v_2=0$ in $\Omega\setminus \Omega'$, we see that for any $x\in \Omega\setminus\Omega'$, $G_{\Omega}(x,\cdot)$ is also a solution of $\cL^* v=0$ in $\Omega'$.
\end{proof}

The following corollary can be obtained by combining Corollary \ref{cor1003} with Lemma \ref{lem3}.

\begin{corollary}\label{cor4}
    Let $R\in(0,\infty)$ and $\kappa\in(0,1)$.
Suppose that $a^{ij}$ are uniformly continuous and $b^i$ are bounded, ensuring that the operator $\cL$ admits a Green's function $G_{R}$ associated with $B_R$.
Then for $x\in B_{\kappa R}$, we have
    \begin{equation*}
        \dashint_{B_R}G_R(x,y)\,dy\ge \nu R^{2-d},
    \end{equation*}
 where $\nu$ depends only on $d$, $\delta$, $\|b\|$, and $\kappa$.
\end{corollary}

\begin{proof}
    Let $x_0\in B_{\kappa R}$ and set $B=B_{(1-\kappa)R}(x_0)\subset B_R$. 
    Let $G_B(x,y)$ be the Green's function of $\cL$ associated with $B$.
    Then $v(x)=\int_{B}G_B(x,y)\,dy$ is the $W_d^2(B)\cap C(\overline{B})$ solution to the equation 
    \[
    \cL v=-1 \quad \text{in} \quad B,
    \]
    with the boundary condition $v=0$ on $\partial B$.
    By Corollary \ref{cor1003}, we have 
    \begin{equation}
        \label{eq_0813_01}
        \int_B G_B(x_0,y)\,dy=v(x_0)\ge \nu (1-\kappa)^2 R^2.
    \end{equation}
    Employing Lemma  \ref{lem3}-($1$) with \eqref{eq_0813_01}, we reach that 
    \[
    \int_{B_R}G_R(x_0,y)\,dy\ge \int_{B}G_R(x_0,y)\,dy\ge \int_{B}G_B(x_0,y)\,dy\ge \nu(1-\kappa)^2R^2.
    \]
    By dividing both sides of the above inequality by $|B_R|$, we finish the proof.
\end{proof}

We are now ready to prove a key estimate: a version of the reverse H\"older inequality for Green's functions.

\begin{proposition}
	\label{prop1002}
Let $\Omega$ be a bounded domain with $\partial \Omega \in C^2$.
Suppose that $a^{ij}$ are uniformly continuous and $b^i$ are bounded, ensuring that the operator $\cL$ admits a Green's function $G_{\Omega}$ associated with $\Omega$.
Then there is a positive constants $K (\ge 1)$ and $N$ which depend on $d$, $\delta$, and $\|b\|$ such that the following holds:\\

\noindent For any $x \in \Omega$, $R>0$ and $x_0\in \Omega$ with $B_{4KR}(x_0)\subset \Omega$,
we have
\begin{equation*}
\begin{split}
    \dashint_{B_R(x_0)} G_\Omega(x,y)^{d/(d-1)}\,dy
    &\le N\left(\dashint_{B_{3KR}(x_0)}G_\Omega(x,y)\,dy\right)^{(d-1)/d}
    \\&\quad+\frac{1}{2}\dashint_{B_{KR}(x_0)} G_\Omega(x,y)^{d/(d-1)}\,dy.        
\end{split}
\end{equation*}
    \end{proposition}
    
    \begin{proof}
We may assume that $x_0=0$.
Let $K \ge 1$ be the one in Proposition \ref{prop0931}.
If $x\notin B_{2KR}$, then the conclusion holds from Lemma \ref{lem3}-($2$) and Proposition \ref{prop0931} (and \ref{cor1} for the assumption of $v$ in Proposition \ref{prop0931}). 
Let $x\in B_{2KR}$ and $G_{3KR}$ be the Green's function of $\cL$ associated with $B_{3KR}$, whose existence is ensured by the assumptions on the coefficients.
By Lemma \ref{lem3}-($1$), we have $G_\Omega(x,\cdot)\ge G_{3KR}(x,\cdot)$ in $B_{3KR}$, which implies 
\begin{equation}
    \label{thm3_eq3}
\begin{split}
\dashint_{B_R}G_\Omega(x,y)^{d/(d-1)}\,dy
&\le 2\dashint_{B_R}\left(G_\Omega(x,y)-G_{3KR}(x,y)\right)^{d/(d-1)}
\,dy 
\\&\quad + 2\dashint_{B_R}G_{3KR}(x,y)^{d/(d-1)}\,dy.
\end{split}
\end{equation}
Note that from Lemma \ref{lem3}-($1$), $G_{\Omega}(x,\cdot)-G_{3KR}(x,\cdot)$ is also a solution of $\cL^* v=0$ in $B_{3KR}$.
Then, by applying Proposition \ref{prop0931} to $G_{\Omega}(x,\cdot)-G_{3KR}(x,\cdot)$, we see that
\[
\dashint_{B_R}\left(G_\Omega(x,y)-G_{3KR}(x,y)\right)^{d/(d-1)}\,dy
\]
\[
\le N\left(\dashint_{B_{KR}}G_\Omega(x,y)-G_{3KR}(x,y)\,dy\right)^{d/(d-1)}
\]
\[+\frac{1}{8}\dashint_{B_{KR}}\left(G_\Omega(x,y)-G_{3KR}(x,y)\right)^{d/(d-1)}\,dy
\]
\begin{equation}\label{thm3_eq1}
\le N\left(\dashint_{B_{KR}}G_\Omega(x,y)\,dy\right)^{d/(d-1)}
+\frac{1}{8}\dashint_{B_{KR}}  G_\Omega(x,y)^{d/(d-1)}\,dy,    
\end{equation}
where the last inequality is due to the nonnegativity of Green's functions.
On the other hand, Corollary \ref{cor1} and Corollary \ref{cor4} imply that
\begin{equation}\label{thm3_eq2}
    \begin{split}
    \dashint_{B_R}G_{3KR}(x,y)^{d/(d-1)}\,dy
    &\le
    N\left(\dashint_{B_{3KR}}G_{3KR}(x,y)\,dy\right)^{d/(d-1)}
\\
&\le
N\left(\dashint_{B_{3KR}}G_{\Omega}(x,y)\,dy\right)^{d/(d-1)},
    \end{split}
\end{equation}
where we use $G_\Omega(x,y)\ge G_{3KR}(x,y)$ to get the last inequality.
By combining \eqref{thm3_eq3}, \eqref{thm3_eq1}, and \eqref{thm3_eq2}, we finish the proof.
\end{proof}

\subsection{\textbf{Proofs of Theorem \ref{cor_main} and Theorem \ref{thm_main}}}
	\label{proof_main_0419}

We now establish the \textit{higher integrability} of Green's functions under regularity assumptions on $a^{ij}$, $b^i$, and the domain $\Omega$. 
This result, stated in Lemma~\ref{lem_0418_1}, forms the foundation for the proofs of our main theorems presented below.

\begin{lemma}[Theorem \ref{thm_main} with smooth data]
    \label{lem_0418_1}
Let $\Omega$ be a bounded domain with $C^2$ boundary.
Assume that $a^{ij}$ are uniformly continuous and $b^i$ are bounded.
Then there exist positive constants $N$ and $q>d/(d-1)$ depending only on $d$, $\delta$, and $\|b\|$ such that the following holds:\\

\noindent The Green's function $G_{\Omega}$ of $\mathcal{L}$ associated with $\Omega$ satisfies
\begin{equation}
	\label{eq_0418_1}
\sup_{x\in \Omega}\left(\int_{\Omega}G_{\Omega}(x,y)^{q}\,dy\right)^{1/q}\le N\operatorname{diam}(\Omega)^{2-d/p},
\end{equation}
where $p$ is the conjugate exponent of $q$, i.e., $1/p + 1/q = 1$.
\end{lemma}

\begin{proof}
By using a translation and scaling argument, we may assume that $\Omega\subset Q_{1/2}$.
We take a domain $\widetilde{\Omega}$ such that $\partial \widetilde{\Omega}\in C^2$ and
\[
\Omega \subset Q_{1/2}\subset Q_1 \subset \widetilde{\Omega} \subset Q_2.
\]
Let $x\in Q_1$.
We rewrite $b$ as $bI_\Omega$ in $\widetilde{\Omega}$.
By Proposition \ref{prop1002}, for any $R>0$ and $x_0 \in Q_1$ satisfying $B_{4KR}(x_0)\subset Q_1 \subset \widetilde{\Omega}$, we have 
\begin{equation*}
\begin{split}
    \dashint_{B_R(x_0)} G_{\widetilde{\Omega}}(x,y)^{d/(d-1)}\,dy
    &\le N\left(\dashint_{B_{3KR}(x_0)}G_{\widetilde{\Omega}}(x,y)\,dy\right)^{(d-1)/d}
    \\&\quad+\frac{1}{2}\dashint_{B_{KR}(x_0)} G_{\widetilde{\Omega}}(x,y)^{d/(d-1)}\,dy,        
\end{split}  
\end{equation*}
where $G_{\widetilde{\Omega}}$ is the Green's function of the operator $\cL$ associated with $\widetilde{\Omega}$, and $N$, $K$ are positive constants from Proposition \ref{prop1002} depending on $d$, $\delta$, and $\|b\|$.
Therefore, it follows that for any $ x_0\in Q_1$, we have  
\begin{multline*}
\left( \cM_{r_0d(x_0)}  G_{\widetilde{\Omega}}( x,\cdot )^{d/(d-1)} \right) \left( x_0 \right)\\
\le N \left[ \left( \cM  G_{\widetilde{\Omega}}(x,\cdot) \right) \left(x_0 \right) \right]^{d/(d-1)}
+\frac{1}{2} \left(  \cM G_{\widetilde{\Omega}}(x,\cdot)^{d/(d-1)} \right) \left( x_0 \right),
\end{multline*}
where $r_0=1/(4K)$ and $d(x_0) = \operatorname{dist}(x_0, \partial Q_1)$.
Using Proposition \ref{reverse}, we see that there exists $q>d/(d-1)$ such that 
\begin{equation*}
\begin{aligned}
    \left(\dashint_{Q_{1/2}}G_{\widetilde{\Omega}}(x,y)^q\,dy\right)^{1/q}
    &\le N\left(\dashint_{Q_1}G_{\widetilde{\Omega}}(x,y)^{d/(d-1)}\,dy\right)^{(d-1)/d} \\
    &\le N\left( \frac{1}{|Q_1|} \int_{\widetilde{\Omega}} G_{\widetilde{\Omega}}(x,y)^{d/(d-1)}\,dy\right)^{(d-1)/d} \\
    &\le N \left( \frac{1}{|Q_1|} \right)^{(d-1)/d} \cdot \operatorname{diam}(\widetilde{\Omega}),
\end{aligned}
    \end{equation*}   
where the last inequality is due to Corollary \ref{cor1}.
On the other hand, from Lemma \ref{lem3}-($1$), we have $G_\Omega(x,y)\le G_{\widetilde{\Omega}}(x,y)$ for all $x, y \in \Omega$.
Thus, we arrive at 
\begin{equation}
	\label{eq250731_1}
\begin{aligned}
\left(\int_{\Omega}G_\Omega(x,y)^q\,dy\right)^{1/q}
&\le \left(\int_{\Omega}G_{\widetilde{\Omega}}(x,y)^q\,dy\right)^{1/q}
\le \left(\int_{Q_{1/2}}G_{\widetilde{\Omega}}(x,y)^q\,dy\right)^{1/q}\\
&= N |Q_{1/2}|^{1/q} \left(\dashint_{Q_{1/2}}G_{\widetilde{\Omega}}(x,y)^q\,dy\right)^{1/q}   \\
&\le N |Q_{1/2}|^{1/q} \left( \frac{1}{|Q_1|} \right)^{(d-1)/d} \cdot \operatorname{diam}(\widetilde{\Omega}) \\ 
&\le N |Q_{1/2}|^{1/q} \left( \frac{1}{|Q_1|} \right)^{(d-1)/d} \cdot \operatorname{diam}(Q_2) \\
& = N(d, \delta, \|b\|)
\end{aligned}
    \end{equation}   
for any $x \in \Omega$, which proves the scaled version of the lemma.\\

\noindent For a general bounded domain $\Omega$, we make the following observation. Let $z_0\in \Omega$ and $R=2\operatorname{diam}(\Omega)$. We define the rescaled domain and coefficients as follows:
\begin{align*}
\Omega_{R,z_0}&=\{x\in\bR^d: z_0+Rx \in \Omega\}\\
a^{ij}_{R,z_0}(x)&=a^{ij}(z_0+Rx)\\
b^i_{R,z_0}(x)&=R\cdot b^i(z_0+Rx)
\end{align*}
The Green's function $G_\Omega(x,y)$ of the operator $\cL=a^{ij}D_{ij}+b^iD_i$ on $\Omega$ is related to the Green's function $G_{R,z_0}$ of the rescaled operator $\cL_{R,z_0}:=a^{ij}_{R,z_0}D_{ij}+b^i_{R,z_0}D_i$ on $\Omega_{R,z_0}$ by the formula:
\[
G_\Omega(x,y)=R^{2-d}\cdot G_{R,z_0}\left( \frac{x-z_0}{R} ,\frac{y-z_0}{R}  \right).
\]
Using this relation and a change of variables, for $x \in \Omega$, we have
\begin{align*}
\left(\int_\Omega G_{\Omega}(x,y)^q\,dy\right)^{1/q}
&=R^{2-d}\left(\int_\Omega G_{R,z_0}\left(\frac{x-z_0}{R} ,\frac{y-z_0}{R} \right)^q\,dy\right)^{1/q}\\
&=R^{2-d+d/q}\left(\int_{\Omega_{R,z_0}} G_{R,z_0}\left(  \frac{x-z_0}{R} ,y \right)^q\,dy\right)^{1/q}\\
&=R^{2-d+d/q}\left(\int_{\Omega_{R,z_0}} G_{R,z_0} \left(  \frac{x-z_0}{R} ,y \right)^q\,dy\right)^{1/q}\\
&\le NR^{2-d+d/q},
\end{align*}
where the last inequality follows from \eqref{eq250731_1}, since $(x-z_0)/R  \in \Omega_{R, z_0} \subset Q_{1/2}$ and $\|b_{R,z_0}\|_{L_d(\Omega_{R,z_0})}= \|b\|_{L_d(\Omega)}\le \|b\|.$
This finishes the proof.
\end{proof}

Now, we apply Lemma \ref{lem_0418_1} to derive the following Aleksandrov-type estimate
\[
\sup_{\Omega} u_{+} \le \sup_{\partial\Omega} u_+ + N  \operatorname{diam}(\Omega)^{2-d/p} \| \left(  \cL u \right)_{-} \|_{L_{p}(\Omega)}
\]
using a standard approximation argument that ensures the existence of Green's functions; see Remark \ref{rmk_Green}.

\begin{proof}[\textbf{Proof of Theorem \ref{cor_main}}]
We begin by fixing a number 
\[
q \in (d/(d-1), d/(d-2))
\]
(with the convention $d/(d-2) := \infty$ when $d = 2$), appeared in Lemma \ref{lem_0418_1}.
We set $d_0 = d_0(d, \delta, \|b\|) \in (d/2, d)$ be the conjugate exponent of $q$.
Note that it is sufficient to prove \eqref{eq1013} for the case $p = d_0$, that is,
\begin{equation}
	\label{d_0}
\sup_{\Omega} u_{+} \le \sup_{\partial\Omega} u_+ + N  \operatorname{diam}(\Omega)^{2-d/d_0} \| \left(  \cL u \right)_{-} \|_{L_{d_0}(\Omega)},
\end{equation}
as the case $p > d_0$ reduces to the case $p=d_0$ by H\"older's inequality.
To prove \eqref{d_0}, we proceed through a series of approximations, reducing the problem to a smooth setting.\\

\noindent \textbf{Step 1: Reduction to a smooth domain.}  
For given $\Omega$, let $\{\Omega_k\}$ be a sequence of bounded $C^2$ domains such that $\Omega_k \nearrow \Omega$ as $k \to \infty$.  
We observe that the conclusion of the theorem for each $\Omega_k$ implies the desired estimate on $\Omega$.
More precisely, if \eqref{d_0} holds for $\Omega_k$ in place of $\Omega$, we have the following:
For any $x \in \Omega$ and $u \in W_{d_0, \mathrm{loc}}^2(\Omega) \cap C(\overline{\Omega})$, there is $k$ such that $x \in \Omega_k$ and
\begin{align*}
u_+(x) &\le \sup_{\partial\Omega_{k}}u_+ + N \operatorname{diam}(\Omega_k)^{2 - d/d_0} \| \left( \cL u \right)_{-} \|_{L_{d_0}(\Omega_k)} \\
 &\le \sup_{\partial\Omega_{k}}u_+ + N(d, \delta, \|b\|) \operatorname{diam}(\Omega)^{2 - d/d_0} \| \left( \cL u \right)_{-} \|_{L_{d_0}(\Omega)},
\end{align*}
where the last inequality is due to $d_0 > d/2$ and $\Omega_k \nearrow \Omega$.
Then by letting $k \to \infty$ we obtain the desired inequality for general bounded $\Omega$.
This step allows us to assume $\Omega$ is sufficiently smooth (\textit{e.g.} $\partial\Omega \in C^{2}$) and assume $u \in W_{d_0}^2(\Omega) \cap C(\overline{\Omega})$, rather than merely being locally in the appropriate function space.\\

\noindent  \textbf{Step 2: Reduction to smooth $u$.}  
We now restrict our attention to the case where $\Omega$ has a $C^2$ boundary and $u \in W_{d_0}^2(\Omega) \cap C(\overline{\Omega})$.
In this setting, there exists a family of smooth functions $\{u^\varepsilon\} \subset C^\infty(\overline{\Omega})$  approximating $u$ in $W_{d_0}^2(\Omega)$, which is continuously embedded into $C(\overline{\Omega})$ since $d_0 > d/2$.
We observe that
\begin{align*}
&\| \cL u - \cL u^{\varepsilon} \|_{L_{d_0}(\Omega)}\\
 &\le N \left( \| D^2 u - D^2 u^{\varepsilon} \|_{L_{d_0}(\Omega)} + \left\| b \left( D u - D u^{\varepsilon}\right) \right\|_{L_{d_0}(\Omega)} \right) \\
& \le N\left( \| D^2 u - D^2 u^{\varepsilon} \|_{L_{d_0}(\Omega)} + \| b\|_{L_d(\Omega)} \left\| D u - D u^{\varepsilon} \right\|_{L_{\frac{dd_0}{d-d_0}}(\Omega)} \right) \\
&\le N \left\| u - u^{\varepsilon} \right\|_{W_{d_0}^2(\Omega)} \to 0 \quad \textrm{as} \,\, \varepsilon \to 0,
\end{align*}
where the last inequality is due to Sobolev embedding and $N$ is independent of $\varepsilon$.
Also, it is clear that
\[
\sup_{\partial \Omega} u^{\varepsilon}_+ \to \sup_{\partial \Omega} u_+ \quad \textrm{as} \,\, \varepsilon \to 0.
\]
Thus, if \eqref{d_0} holds for $u^{\varepsilon}$ for each $\varepsilon > 0$ (with smooth $\Omega$), by letting $\varepsilon \to 0$, the same conclusion holds for general $u \in W_{d_0}^2(\Omega) \cap C(\overline{\Omega})$.
This, together with Step~1, allows us to assume $u \in C^{\infty}(\overline{\Omega})$ and $\partial\Omega \in C^2$.\\

\noindent  \textbf{Step 3: Reduction to regular $a^{ij}$ and $b^i$.} 
Now we focus on the case where $\Omega$ has a $C^2$ boundary and $u \in C^{\infty}(\overline{\Omega})$.
We approximate the coefficients $a^{ij}$ and $b^i$ by
\[
a^{ij}_{\varepsilon} := (a^{ij})^{(\varepsilon)} \quad \textrm{and} \quad b^i_{\varepsilon} := (b^i  I_{\Omega})^{(\varepsilon)},
\]
where $f^{(\varepsilon)}$ means the standard mollification of a locally integrable function $f$.
From the basic properties of mollification, we see that $a^{ij}_{\varepsilon}$ are uniformly continuous and $b^i_{\varepsilon}$ are bounded.
We also note that $a^{ij}_{\varepsilon}$ satisfy the ellipticity condition with same $\delta$, and $\|b_{\varepsilon}\|_{L_d(\Omega)} \le \|b\|$, where $b_{\varepsilon} = (b_{\varepsilon}^1, \ldots, b_{\varepsilon}^d)$.
Moreover, since $d_0 < d$,
\[
a^{ij}_{\varepsilon} \to a^{ij}, \quad b^{i}_{\varepsilon} \to b^{i} \quad \textrm{in} \,\, L_{d_0}(\Omega)
\]
as $\varepsilon \to 0$, and then, having in mind that $u \in C^{\infty}(\overline{\Omega})$, it follows that
\[
\cL_{\varepsilon} u := a^{ij}_{\varepsilon}D_{ij}u + b^{i}_{\varepsilon}D_i u \to \cL u \quad \textrm{in} \,\, L_{d_0}(\Omega)
\]
as $\varepsilon \to 0$.
Thus, if \eqref{d_0} holds for $\cL_{\varepsilon}$ for each $\varepsilon > 0$ (with smooth $\Omega$ and $u$), by letting $\varepsilon \to 0$, the same conclusion holds for general $\cL$.
\\

\noindent  \textbf{Step 4: The smooth case.}  
Due to the above reductions, it suffices to prove \eqref{d_0} with further assumptions that $\partial\Omega \in C^2$, $u \in C^\infty(\overline{\Omega})$, $a^{ij}$ are uniformly continuous, and $b^i$ are bounded.
From Remark \ref{rmk_Green}, the Green's function $G_\Omega$ exists and 
\[
w(x) := \int_{\Omega} G_{\Omega}(x, y) \, (\cL u(y))_{-}\,dy, \quad \text{for $x\in \Omega$,}
\]
solves $\cL w = -(\cL u)_-$ in $\Omega$ and $w = 0$ on $\partial\Omega$.
It follows from \eqref{eq_0418_1} that
\[
\sup_{\Omega} w_{+} \le N(d, \delta, \|b\|) \operatorname{diam}(\Omega)^{2 - d/d_0}\| (\cL u)_{-}\|_{L_{d_0}(\Omega)}.
\]
On the other hand, by Theorem \ref{thm_abp}, we have
\[
\sup_{\Omega} v_{+} \le \sup_{\partial\Omega}  v_{+} = \sup_{\partial\Omega}  u_{+},
\]
where $v = u - w$.
Combining the above inequalities yields the desired estimate \eqref{d_0}.  
This completes the proof of the theorem.
\end{proof}

We now drop the regularity assumptions on \( a^{ij} \), \( b^i \), and \( \partial \Omega \) in the statement of Lemma \ref{lem_0418_1}, and assume only the existence of the Green's function. That is, we proceed to prove Theorem \ref{thm_main}.

\begin{proof}[\textbf{Proof of Theorem \ref{thm_main}}]
Let $d_0=d_0(d,\delta,\|b\|)\in (d/2,d)$ be the constant from Theorem \ref{cor_main}.
We use the duality argument.
Let $f$ be a nonnegative function in $C_c^\infty(\Omega)$.
Then, by the definition of the Green's function, we see that
\[
u(x) := \int_{\Omega} G_{\Omega}(x, y) \, f(y)\,dy, \quad \text{for $x\in \Omega$,}
\]
is in $W_{d,\mathrm{loc}}^2(\Omega)\cap C(\overline{\Omega})\subset W_{d_0,\mathrm{loc}}^2(\Omega)\cap C(\overline{\Omega})$ and solves $\cL u = -f$ in $\Omega$ and $u = 0$ on $\partial\Omega$.
Then, by applying the estimate \eqref{eq1013} to $u$, we obtain that 
\[
\sup_{x\in \Omega}\int_{\Omega} G_{\Omega}(x, y) \, f(y)\,dy\le N \operatorname{diam}(\Omega)^{2-d/d_0}\|f\|_{L_{d_0}(\Omega)}.
\]
Combining the above with the fact that $G_\Omega\ge0$, we reach to
\[
\sup_{x\in \Omega}\left(\int_{\Omega}G_{\Omega}(x,y)^{q}\,dy\right)^{1/q}\le N\operatorname{diam}(\Omega)^{2-d/d_0},
\]
where $q$ is the conjugate exponent of $d_0$. Thus, the desired result has been established.
\end{proof}

\section*{Acknowledgement}
P. Jung has been supported by the National Research Foundation of Korea (NRF) funded by the Korean government (MSIT) (2019R1A2C1084683 and RS-2022-NR069609).
K. Woo has been supported by the National Research Foundation of Korea (NRF) grant funded by the Korean government (MSIT) (No.2019R1A2C1084683 and RS-2022-NR070754).
The authors would like to express their sincere gratitude to the anonymous referee for their valuable comments and suggestions, which greatly helped improve the clarity of the paper.


\begin{thebibliography}{10}

\bibitem{MR0199540}
A.~D. Aleksandrov.
\newblock Majorants of solutions of linear equations of order two.
\newblock {\em Vestnik Leningrad. Univ.}, 21(1):5--25, 1966.

\bibitem{A58}
A.~D. Aleksandrov.
\newblock Dirichlet’s problem for the equation {$\det || z_{ij}||= \phi (z_1, \ldots , z_n, z, x_1, \ldots , x_n)$}. I.
\newblock {\em Vestnik Leningrad. Univ. Ser. Mat. Meh. Astr}, 13(1):5--24, 1958.

\bibitem{MR0126604}
I.~J. Bakel'man.
\newblock On the theory of quasilinear elliptic equations.
\newblock {\em Sibirsk. Mat. \v{Z}.}, 2:179--186, 1961.

\bibitem{Cab1995}
X. Cabr\'e.
\newblock On the Aleksandrov--Bakel'man--Pucci estimate and the reversed H\"older inequality for solutions of elliptic and parabolic equations.
\newblock {\em Comm. Pure Appl. Math.}, 48(5):539--570, 1995.


\bibitem{MR1483890}
R.~F. Bass.
\newblock {\em Diffusions and elliptic operators}.
\newblock Probability and its Applications (New York). Springer-Verlag, NewYork, 1998.

\bibitem{CNS84}
L.~A. Caffarelli, L. Nirenberg, and J. Spruck.
\newblock The Dirichlet problem for nonlinear second-order elliptic equations I. Monge-Amp{\'e}re equation.
\newblock {\em Comm. Pure Appl. Math.}, 37(3):369--402,
  1984.

\bibitem{MR1005611}
L.~A. Caffarelli.
\newblock Interior a priori estimates for solutions of fully nonlinear
  equations.
\newblock {\em Ann. of Math. (2)}, 130(1):189--213, 1989.

\bibitem{caffarelli1995fully}
L.~A. Caffarelli and X. Cabr{\'e}.
\newblock {\em Fully nonlinear elliptic equations}, volume~43.
\newblock American Mathematical Society, 1995.

\bibitem{Coifman1974}
R. Coifman and C. Fefferman.
\newblock Weighted norm inequalities for maximal functions and singular integrals.
\newblock {\em Stud. Math.}, 51(3):241--250, 1974.

\bibitem{DK2022}
H. Dong and N.~V. Krylov
\newblock Aleksandrov's estimates for elliptic equations with drift in a Morrey spaces containing {$L_d$}.
\newblock {\em Proc. Am. Math. Soc.}, 150, No. 4, 1641--1645, 2022.

\bibitem{MR1237053}
L. Escauriaza.
\newblock {$W^{2,n}$} a priori estimates for solutions to fully nonlinear
  equations.
\newblock {\em Indiana Univ. Math. J.}, 42(2):413--423, 1993.

\bibitem{MR771392}
E.~B. Fabes and D.~W. Stroock.
\newblock The {$L^p$}-integrability of {G}reen's functions and fundamental solutions for elliptic and parabolic equations.
\newblock {\em Duke Math. J.}, 51(4):997--1016, 1984.


\bibitem{MR1632776}
K. Fok.
\newblock A nonlinear {F}abes-{S}troock result.
\newblock {\em Comm. Partial Differ. Equ.}, 23(5-6):967--983, 1998.

\bibitem{10.1007/BF02392268}
F.~W. Gehring.
\newblock The {$L_p$}-integrability of the partial derivatives of a quasiconformal
  mapping.
\newblock {\em Acta Math.}, 130(none):265--277, 1973.

\bibitem{MR0717034}
M. Giaquinta.
\newblock {\em Multiple integrals in the calculus of variations and nonlinear
  elliptic systems}, volume 105 of {\em Annals of Mathematics Studies}.
\newblock Princeton University Press, Princeton, NJ, 1983.

\bibitem{MR0737190}
D. Gilbarg and N.~S. Trudinger.
\newblock {\em Elliptic partial differential equations of second order}, volume
  224 of {\em Grundlehren der mathematischen Wissenschaften [Fundamental
  Principles of Mathematical Sciences]}.
\newblock Springer-Verlag, Berlin, second edition, 1983.

\bibitem{MR4515258}
S. Koike and A. Swiech.
\newblock Aleksandrov-{B}akelman-{P}ucci maximum principle for {$L^p$}-viscosity solutions of equations with unbounded terms.
\newblock {\em J. Math. Pures Appl. (9)}, 168:192--212, 2022.

\bibitem{MR2435520}
N.~V. Krylov.
\newblock {\em Lectures on elliptic and parabolic equations in Sobolev spaces}, volume 96 of {\em Graduate Studies in Mathematics}.
\newblock American Mathematical Society, Providence, RI, 2008. 

\bibitem{MR4317707}
N.~V. Krylov.
\newblock On stochastic equations with drift in {$L_d$}.
\newblock {\em Ann. Probab.}, 49(5):2371--2398, 2021.

\bibitem{MR4252191}
N.~V. Krylov.
\newblock On stochastic {I}t\^{o} processes with drift in {$L_ d$}.
\newblock {\em Stochastic Process. Appl.}, 138:1--25, 2021.

\bibitem{Krylov2020}
N.~V. Krylov.
\newblock Linear and fully nonlinear elliptic equations with {$L_d$}-drift.
\newblock {\em Comm. Partial Differ. Equ.},
  45(12):1778--1798, 2020.

\bibitem{MR563790}
N.~V. Krylov and M.~V. Safonov.
\newblock A property of the solutions of parabolic equations with measurable
  coefficients.
\newblock {\em Izv. Akad. Nauk SSSR Ser. Mat.}, 44(1):161--175, 239, 1980.


\bibitem{MR2064654}
G.~M. Lieberman.
\newblock Maximum estimates for oblique derivative problems with right hand side in {$L^p,\ p<n$}.
\newblock {\em Manuscripta Math.}, 112(4):459--472, 2003.

\bibitem{MR0214905}
C. Pucci.
\newblock Limitazioni per soluzioni di equazioni ellittiche.
\newblock {\em Ann. Mat. Pura Appl. (4)}, 74:15--30, 1966.

\bibitem{MR2667641}
M.~V. Safonov.
\newblock Non-divergence elliptic equations of second order with unbounded drift.
\newblock In {\em Nonlinear partial differential equations and related topics},
  volume 229 of {\em Amer. Math. Soc. Transl. Ser. 2}, pages 211--232. Amer.
  Math. Soc., Providence, RI, 2010.

\bibitem{TU83}
N.~S. Trudinger and John~I.E. Urbas.
\newblock The Dirichlet problem for the equation of prescribed Gauss curvature.
\newblock {\em Bull. Aust. Math. Soc.},
  28(2):217--231, 1983.

\bibitem{MR2486925}
N. Winter.
\newblock {$W^{2,p}$} and {$W^{1,p}$}-estimates at the boundary for solutions
  of fully nonlinear, uniformly elliptic equations.
\newblock {\em Z. Anal. Anwend.}, 28(2):129--164, 2009.

\end{thebibliography}
\end{document}